\documentclass[7pt,a4paper,twocolumn]{article}
\usepackage[utf8]{inputenc}
\usepackage[T1]{fontenc}
\usepackage{mathptmx}
\usepackage[pdftex]{graphicx}
\usepackage[pdftex,linkcolor=black,pdfborder={0 0 0}]{hyperref}
\usepackage{calc}
\usepackage{enumitem}

\usepackage[a4paper, lmargin=0.08\paperwidth, rmargin=0.08\paperwidth, tmargin=0.08\paperheight, bmargin=0.08\paperheight]{geometry}

\usepackage[all]{nowidow}
\usepackage[protrusion=true,expansion=true]{microtype}

\usepackage{amsmath}
\usepackage{amsthm}
\newtheorem{theorem}{Theorem}
\usepackage{amsfonts}
\usepackage{complexity}

\usepackage{titlesec}
\titlespacing*{\section}{0pt}{1pt}{1pt}
\titlespacing*{\subsection}{0pt}{1pt}{1pt}
\titlespacing*{\paragraph}{0pt}{1pt}{1pt}

% ...
\usepackage{lipsum}
\usepackage[capitalize]{cleveref}

\newtheorem{problem}[theorem]{Problem}

\newcommand{\calH}{\ensuremath{\mathcal{H}}}
\newcommand{\calE}{\ensuremath{\mathcal{E}}}

\author{Ueckerdt, Torsten}
\title{A Note on Polychromatic Colorings of Shift-Chains}
\makeatletter

\hypersetup{
    pdfsubject = {Polychromatic Colorings of Shift-Chains},
    pdftitle = {\@title},
    pdfauthor = {\@author}
}

\usepackage{fancyhdr}
\begin{document}

We popularize the question whether, for $m$ large enough, all $m$-uniform shift-chain hypergraphs are properly $2$-colorable.
On the other hand, we show that for every $m$ some $m$-uniform shift-chains are not polychromatic $3$-colorable.

\section*{Vertex-Colorings of Shift-Chains}

An \emph{ordered $m$-uniform} hypergraph $\calH = (V,\calE)$ has vertex-set $V = [n]$ and edge-set $\calE \subseteq V^m$, where $a_1 < \cdots < a_m$ for every $A = (a_1,\ldots,a_m) \in \calE$.
In other words, vertices are linearly ordered (numbered) and edges are $m$-tuples (vectors) of vertices in increasing order.

If for two edges~$A = (a_1,\ldots,a_m)$ and $B = (b_1,\ldots,b_m)$ it holds $a_i \leq b_i$ for all $i \in [m]$, then we write $A \preceq B$ and call $A,B$ \emph{comparable}.
Then $\calH = (V,\calE)$ is a \textbf{shift-chain} if any two edges in $\calE$ are comparable.
Observe that this implies $|\calE| \leq m(|V|-m)+1$.
In particular, $2$-uniform shift-chains (equivalently, graphs of queue number~$1$) are $2$-degenerate and hence properly\footnote{no edge is monochromatic, i.e., has all vertices of the same color} $3$-colorable.

Consequently, every $m$-uniform shift-chain is properly $3$-colorable, by considering its restriction to (any) two coordinates.
Indeed, restricting each $m$-tuple edge in $\calH$ to the $2$-tuple of its first two coordinates gives a $2$-uniform shift-chain, each of whose proper colorings is also a proper coloring of $\calH$.

\smallskip

P{\'a}lv{\"o}lgyi introduces shift-chains in his PhD thesis~\cite{Pal10a} (see also the survey~\cite{PPT13} of Pach et al.) and asks:

\begin{problem}[P{\'a}lv{\"o}lgyi 2010~\cite{Pal10a}, Pach et al. 2013~\cite{PPT13}]{\ \\}\label{prob:2-colorable}
    \hspace*{-4pt}%
    Is there an $m \in \mathbb{N}$ such that every $m$-uniform shift-chain is properly $2$-colorable?
\end{problem}

In fact $m\geq 4$ is necessary for a positive answer of \cref{prob:2-colorable}.
A $3$-uniform, non-$2$-colorable shift-chain due to Radoslav Fulek~\cite{Pal10a,PPT13} (slightly simplified) is shown in \cref{fig:3-uniform-LB}.
In these figures, putting vertices as columns and coordinates as rows, comparable edges correspond to non-crossing polylines.

\begin{figure}[ht]
    \centering
    \includegraphics[page=1]{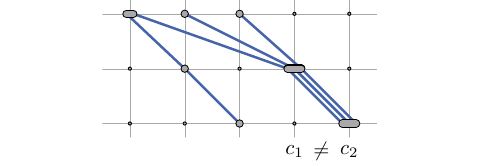}
    
    \vspace{15pt}
    
    \includegraphics[page=2]{3-uniform-not-2-colorable_ORCS.pdf}
    \caption{
        Top: Forcing different colors in any proper \mbox{$2$-coloring}.
        Bottom: Combining four copies of the $3$-uniform shift-chain above to force a monochromatic (black, thick) edge.
    }
    \label{fig:3-uniform-LB}
\end{figure}

While~\cref{prob:2-colorable} remains open, the stronger property of \textbf{polychromatic $3$-colorability} does not hold.
That is, we cannot $3$-color the vertices such that every edge contains vertices of all three colors.

\begin{theorem}\label{thm:polychromatic-3-coloring}
    For all $m \in \mathbb{N}$, some $m$-uniform shift-chains have no polychromatic $3$-coloring.
\end{theorem}
\begin{proof}
    We give an inductive construction for the desired $m$-uniform shift-chains $\calH_m = ([n_m], \calE_m)$.
    For $m = 1$ it is enough to take $\calH_1 = ([n_1],\calE_1)$ with $n_1 = 1$ and $\calE_1 = \{ (1) \}$.

    Given $\calH_m = ([n_m], \calE_m)$ for some $m \geq 1$, let $A_1 \preceq \cdots \preceq A_t$ be the total ordering of edges in $\calH_m$.
    Now we take 
    \begin{align*}
        n_{m+1} &= n_m + m\cdot t + 1\\
        \calE' &= \{ A_{i,j} = (A_i, n_m + (i-1)\cdot m + j) \mid i \in [t], j \in [m]\}\\
        \calE'' &= \{B_i = (n_m + im + 1,\ldots, n_m+im + m ,n_{m+1}) \mid i \in [t]\}\\
        \calH_{m+1} &= ([n_{m+1}],\calE_{m+1} = \calE' \cup \calE'').
    \end{align*}
    Observe that $\calH_{m+1}$ is indeed an $(m+1)$-uniform shift-chain.
    See \cref{fig:m-uniform-LB} for an illustration.

    \begin{figure}[ht]
        \centering
        \includegraphics{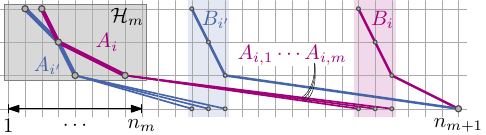}
        \caption{Constructing $\calH_{m+1}$ from $\calH_m$.}
        \label{fig:m-uniform-LB}
    \end{figure}

    Now, consider any $3$-coloring of $\calH_{m+1}$, i.e., of $[n_{m+1}]$.
    This is also a $3$-coloring of $[n_m] \subset [n_{m+1}]$ and hence of~$\calH_m$.
    By induction, this coloring is not polychromatic for~$\calH_m$.
    Without loss of generality let edge $A_i \in \calE_m$ have no vertex of color~$1$.
    Then, either one of $A_{i,1},\ldots,A_{i,m} \in \calE'$ has no vertex of color~$1$, or all vertices in $B_i$ except $n_{m+1}$ are colored~$1$ and hence $B_i \in \calE''$ has no vertex of color~$2$ or no vertex of color~$3$.
    In any case, our $3$-coloring is not polychromatic for $\calH_{m+1}$, as desired.
\end{proof}

Let us remark that, using the \emph{Union Lemma} of~\cite{DP23}, one can strengthen \cref{thm:polychromatic-3-coloring} if one could find for every $m \in \mathbb{N}$ two \mbox{$m$-uniform} shift-chains $\calH_1 = ([n],\calE_1)$, $\calH_2 = ([n],\calE_2)$ such that their union $\calH = ([n], \calE_1 \cup \calE_2)$ is not properly $3$-colorable.

\bibliographystyle{plainurl}
\bibliography{literature.bib}
\end{document}